\documentclass{amsart}

\usepackage[utf8]{inputenc}
\usepackage[T1]{fontenc}

\usepackage{url,textcmds,amssymb}
\usepackage{mathtools,amscd,enumerate}
\usepackage{tikz}
\usepackage[main=british]{babel}

\usetikzlibrary{patterns}

\newtheorem{theorem}{Theorem}[section]
\newtheorem{proposition}[theorem]{Proposition}
\newtheorem{lemma}[theorem]{Lemma}

\theoremstyle{definition}

\theoremstyle{remark}
\newtheorem{remark}[theorem]{Remark}
\newtheorem{example}[theorem]{Example}
\newtheorem*{acknowledgement}{Acknowledgement}

\newcommand{\itref}[1]{\textup{(\ref{#1})}}

\newcommand{\lie}[1]{\mathfrak{#1}}

\newcommand{\n}{\lie n}

\newcommand{\tf}{\lie t}
\newcommand{\z}{\lie z}

\newcommand{\bC}{\mathbb C}
\newcommand{\bH}{\mathbb H}

\newcommand{\bL}{\mathbb L}
\newcommand{\bN}{\mathbb N}
\newcommand{\bP}{\mathbb P}
\newcommand{\bR}{\mathbb R}
\newcommand{\bZ}{\mathbb Z}

\newcommand{\C}{\bC}
\newcommand{\R}{\bR}

\newcommand{\hkq}{{/\mkern-3mu/\mkern-3mu/}}

\DeclarePairedDelimiter{\abs}{\lvert}{\rvert}
\DeclarePairedDelimiterX{\inp}[2]{\langle}{\rangle}{#1, #2} 
\DeclarePairedDelimiter{\norm}{\lVert}{\rVert}

\newcommand{\with}{}
\newcommand{\SetSymbol}[1][]{\nonscript\:#1\vert
  \allowbreak\nonscript\:\mathopen{}}
\DeclarePairedDelimiterX{\Set}[1]{\{}{\}}{%
  \renewcommand{\with}{\SetSymbol[\delimsize]}
  #1 }

\DeclareMathOperator{\imp}{Im} 
\DeclareMathOperator{\imm}{im} 

\DeclareMathOperator{\nLie}{Lie}

\DeclareMathOperator{\stab}{stab}

\newcommand{\iH}{\imp\bH}

\begin{document}

\title{Hypertoric manifolds of infinite topological type}

\author{Andrew Dancer}
\address[Dancer]{Jesus College\\
Oxford, OX1 3DW\\
United Kingdom} \email{dancer@maths.ox.ac.uk} 

\begin{abstract}
We analyse properties of hypertoric manifolds of infinite topological 
type, including their topology and complex structures. We show that our
manifolds have the homotopy type of an infinite union of compact 
toric varieties. We also discuss hypertoric analogues of the
periodic Ooguri-Vafa spaces.
\end{abstract}

\maketitle

\section{Introduction}
\label{sec:introduction}

A hypertoric manifold is a hyperk\"ahler manifold of real
dimension $4n$ with a tri-Hamiltonian action of a torus $T^n$ 
of dimension $n$. (In the literature completeness is often assumed, but
in some cases we will drop this assumption).
Complete examples were systematically constructed and
 analysed in \cite{Bielawski:tri-Hamiltonian},
 \cite{BD} in the case
when they have finite topological type, meaning that all
Betti numbers are finite. These particular examples
 generalise the four-dimensional
Gibbons-Hawking spaces \cite{Gibbons-H:multi} and the higher-dimensional 
examples of Goto \cite{Goto}, and retract onto a finite union of
toric varieties.

Recently, the complete four-dimensional examples, without restriction on their topology, were classified in \cite{Swann:twist-mod}. Examples of infinite topological
type had been given by 
Anderson,Kronheimer and LeBrun \cite{AndersonMT-KL:infinite} 
and have also been studied by Hattori \cite{Hattori:Ainfty-volume}, building
on work of Goto \cite{Goto:A-infinity}, that also included examples
in higher dimensions. The approach of Goto was generalised in a  systematic
construction in higher dimensions by the
author and Swann in our paper \cite{DS-SMS}, using
hyperK\"ahler quotients of certain Hilbert manifolds. This was then sufficient
to obtain the full classification of complete hypertoric manifolds by combining
the results of \cite{Bielawski:tri-Hamiltonian}, \cite{Swann:twist-mod}.


As in the finite case, much of the geometry and topology is
encoded by a configuration of codimension 3 affine subspaces ({\em
flats}) in $\bR^{3n}$, generalising the points in $\bR^3$ that are the
centres of the Gibbons-Hawking metrics. In
\cite{BD} there were only finitely many such
flats, but we shall now consider the situation where there are
infinitely many of them. These need to be
 spaced out suitably to make certain sums
converge.  

The purpose of this paper is to further analyse the geometries obtained
in the case of infinite topological type. In particular, we describe their
homotopy type and their structure as complex manifolds.
The topology is now generated by an infinite collection of
compact toric varieties of restricted types. The $T^n$ action still induces a moment map
that surjects onto $\bR^{3n}$. The fibres are quotient tori ($T^n$ for generic
fibres), whose dimension is controlled by the intersection properties
of the flats.  The Pedersen-Poon generalisation of the Gibbons-Hawking
calculations enables us to write the metric in terms of a generalised
monopole.  The explicit formulae now involve an infinite sum involving
the distances from the flats.  

We also briefly discuss generalisations
of the 4-dimensional (incomplete) periodic Ooguri-Vafa metric to the
hypertoric set-up.

\section{Construction of hypertoric manifolds}
\label{sec:constr-hypert-manif}

We begin by reviewing our construction from \cite{DS-SMS} which produces
certain hypertoric manifolds $M(\beta, \lambda)$ 
(including examples of infinite topological type)
via a hyperk\"ahler quotient of a flat Hilbert manifold.

We let \( \bL \) denote a finite or  
countably infinite set and let $\bH = \bR^{4}$ be the quaternions.
  Given
\( \Lambda = (\Lambda_k)_{k\in\bL} \in \bH^\bL \), we define
\( \lambda = (\lambda_k)_{k\in \bL} \in {\rm Im} \; \bH^\bL\) by
\( \lambda_k = - \tfrac12\overline\Lambda_k i \Lambda_k \in \iH \).
Put
  \( L^2(\bH) = \Set[\Big]{ v \in
  \bH^{\bL} \with \sum_{k\in\bL} \abs{v_k}^2 < \infty} \) and equip the 
 Hilbert manifold \( M_\Lambda = \Lambda + L^2(\bH) \),
with the flat hyperKähler structure induced by the
\( L^2 \)-metric and the complex structures obtained by regarding
\( L^2(\bH) \) as a right \( \bH \)-module.

For the following
construction, we require $\lambda$ to satisfy the  growth condition
\begin{equation}
  \label{eq:convergence}
  \sum_{k\in \bL}(1 + \abs{\lambda_k})^{-1} < \infty.
\end{equation}

Consider the Hilbert group
\begin{equation*}
  T_\lambda = \Set[\Big]{ g \in T^{\bL} =
  (S^1)^{\bL} \with \sum_{k\in\bL} (1 +
  \abs{\lambda_k})\, \abs{1-g_k}^2 < \infty}.
\end{equation*}


\noindent whose Lie algebra ${\lie t}_{\lambda}$  consists of
those $t \in \bR^\bL$ such that\\
$\sum_{k\in\bL} (1 + \abs{\lambda_k})\, \abs{t_k}^2$ is finite.
Note that $|g_k h_k -1| \leq |(g_k-1)(h_k-1)| + |g_k-1| + |h_k -1 |$
and $|g_k^{-1} - 1| = |g_k|^{-1} | g_k - 1|$, so as $g_k \rightarrow 1$
as $|k| \rightarrow \infty$, the group axioms do indeed hold for
$T_\lambda$.

The group \( T_\lambda \) acts on \( M_\Lambda \) via
\( gx = (g_kx_k)_{k\in\bL} \). To see this,
observe that for \( g \in T_\lambda \) and
\( x = \Lambda + v \in M_\Lambda \), we have
\( gx = g\Lambda + gv = \Lambda - (1-g)\Lambda + gv \), but
\( gv \in L^2(\bH) \), since $g_k \rightarrow 1$ as $|k| \rightarrow \infty$,
 and \( \norm{(1-g)\Lambda}^2 = \sum_{k\in\bL}
\tfrac12\abs{\lambda_k}\abs{1-g_k}^2 \leqslant \tfrac12 \sum_{k\in\bL}
(1+\abs{\lambda_k}) \abs{1-g_k}^2 \), which is finite by the
definition of~\( T_\lambda \), so \( (1-g)\Lambda \in L^2(\bH) \)
too.  The action preserves the flat hyperKähler structure 
on $M_\Lambda$, and has moment map $\mu_{\Lambda} : M_{\Lambda}
\rightarrow {\rm Im} \; \bH \otimes {\lie t}_{\lambda}^*$ given by
\begin{equation*}
  \mu_\Lambda(x) = \sum_{k\in\bL} \bigl(\lambda_k + \tfrac12\overline
  x_k\mathbf ix_k \bigr)e_k^*.
\end{equation*}
where \( e_i = (\delta_i^k)_{k\in\bL} \) with
 \( \delta_i^k \)  Kronecker's delta and $e_k^*(e_i) = \delta_i^k$.
The choice of constant $\lambda_k$ is to cancel the quadratic
terms in $\Lambda$ when we write $x = \nu + \Lambda$ with $\nu
\in L^2 (\bH)$. The cross terms that are linear in $\Lambda$
give a convergent sum of terms such as $\bar{v}_k \mathbf i t_k \Lambda_k$
when paired with $t \in \tf_{\lambda}$, as both $(\nu_k)$ and $(t_k \Lambda_k)$
are square-summable.

\medskip
Our hypertoric manifolds are produced as hyperk\"ahler quotients
of $M_{\Lambda}$ by finite-codimension subtori of $T_{\Lambda}$, defined
as follows.
For each \( k \in \bL \), we let \( u_k \in \bZ^n \) be a non-zero 
{\em primitive} vector, meaning that for all integers
$m>1$, we have $u_{k} \notin m \bZ^n$.
We also assume that $\{u_k : k \in \bL\}$ spans $\bR^n$.
Now consider the linear map
\( \beta \colon \lie t_\lambda \to \bR^n \) given by
\( \beta(e_k) = u_k \), 

\medskip
 Supposing \( \beta \)~is continuous, then we define
\( \lie n_\beta = \ker \beta \subset \lie t_\lambda\), so
\( \lie n_\beta \) is a closed subspace of \( \lie t_{\lambda} \).  
(We shall shortly give sufficient conditions for continuity of $\beta$). As
we are taking the \( u_k \) to be integral, we may
define a Hilbert subgroup \( N_\beta \) of~\( T_\lambda \) by
\begin{equation*}
  N_\beta = \ker(\exp \circ \beta \circ \exp^{-1} \colon T_\lambda \to
  T^n).
\end{equation*}
This gives exact sequences
\begin{gather}
  \label{eq:ex-lie}
  \begin{CD}
    0 @>>> \lie n_\beta @>\iota>> \lie t_\lambda @>\beta>> \bR^n @>>>
    0,
  \end{CD}\\
  \label{eq:ex-lie-dual}
  \begin{CD}
    0 @>>> (\bR^n)^{*} @> \beta^* >> \lie t_\lambda^{*} @> \iota^* >> \lie n_\beta^* @>>> 0,
  \end{CD}
\end{gather}
and also the exact sequence of groups 
\[
1 \rightarrow N_{\beta} \rightarrow T_{\lambda} \rightarrow T^n 
= T_{\lambda}/N_{\beta} \rightarrow 1.
\]
The hyperKähler moment map for the subgroup \( N_\beta \) is then
\( \mu_\beta = \iota^*\mu_\Lambda \colon M_\Lambda \to \iH \otimes
\lie n_\beta^* \).  We define the hyperk\"ahler quotient
\begin{equation*}
  M = M(\beta,\lambda) := M_\Lambda \hkq N_\beta = \mu_\beta^{-1}(0)
  / N_\beta.
\end{equation*}

Since \eqref{eq:ex-lie-dual} is exact, we have
\( \ker\iota^* = \imm\beta^* \). Now $\beta^*(s) (e_k) = s(\beta(e_k)) =
s(u_k)$ for $s \in (\R^n)^*$, so $\beta^*(s) = \sum s(u_k) e_k^*$.
 Hence a point \( x \in M_\Lambda \) lies in the zero set of the
  hyperKähler moment map~\( \mu_\beta \) for~\( N_\beta \) if and only
  if there is an \( a \in \iH \otimes (\bR^n)^* \) with
  \begin{equation}
    \label{eq:im-beta}
    a(u_k) = \lambda_k + \tfrac12\overline x_k\mathbf ix_k
  \end{equation}
  for each \( k \in \bL \), where \( u_k = \beta(e_k) \).  \qed

 As in
\cite{BD}, the codimension 3 affine subspaces
\( H_k \subset \iH \otimes (\bR^n)^* \cong \bR^{3n} \) 
\begin{equation}
  \label{eq:flat}
  H_k = H(u_k,\lambda_k)
  = \Set{a \in \iH \otimes (\bR^n)^* \with a(u_k) = \lambda_k}
\end{equation}
are called \emph{flats}. Equation(\ref{eq:im-beta}) shows that
 \( x_k = 0 \) if and only if
\( a(u_k) = \lambda_k \), that is, if and only if $a \in H_k$.

The following result was proved by the author and Swann in \cite{DS-SMS}.
\begin{theorem}
  \label{thm:construction} \cite{DS-SMS}
  Suppose \( u_k = \beta(e_k) \in \bZ^n \), \( k \in \bL \), are
  primitive and span~\( \bR^n \).  Let \( \lambda_k \in \iH \),
  \( k \in \bL \), be given such that the convergence
  condition~\eqref{eq:convergence} holds and the flats
  \( H_k = H(u_k,\lambda_k) \), \( k \in \bL \), are distinct.  Then
  the hyperKähler quotient \( M = M(\beta,\lambda) \) is smooth 
 and of dimension $4n$ if
  \begin{enumerate}[\upshape (a)]
  \item\label{item:a} any set of \( n+1 \) flats~\( H_k \) has empty
    intersection, and
  \item\label{item:b} whenever \( n \) distinct flats
    \( H_{k(1)},\dots,H_{k(n)} \) have non-empty intersection the
    corresponding vectors \( u_{k(1)},\dots,u_{k(n)} \) form a
    \( \bZ \)-basis for~\( \bZ^n \).
  \end{enumerate}
\end{theorem}

The two conditions ensure that
 $N_{\beta}$ acts freely on $\mu_{\beta}^{-1}(0)$. Note that we have
a residual hyperK\"ahler action of $T^n = T_{\Lambda}/N_{\beta}$
on the $4n$-manifold $M(\beta, \lambda)$, justifying the terminology
{\em hypertoric}.

The proof involved the following result, which is useful in its own right.

\begin{proposition} \label{finite} Suppose \( u_k \in \bZ^n \),
  \( k\in\bL \), are primitive, span~\( \bR^n \) and satisfy
  condition~\itref{item:b} of Theorem~\ref{thm:construction}.  Then
  \( \mathcal U = \Set{u_k\with k\in\bL} \) is finite.
\end{proposition}

\medskip 
Let us observe (as in the proof of Theorem 3.3 in
\cite{BD}) that conditions (a), (b) imply that if
$J$ is a maximal set of indices such that $\cap_{k \in J} H_k$ is
non-empty, then $\{ u_k : k \in J \}$ is a $\bZ$-basis for $\bZ^n$. For
maximality implies that every other $u_i$ is in the span of
$\{ u_k : k \in J \}$, so, as we are assuming that the
$u_k : k \in \bL$ span $\bR^n$, then $\{ u_k : k \in J \}$ is also
spanning. Condition (a) now implies that $|J| = n$, and condition (b)
implies that $\{ u_k : k \in J \}$ is a $\bZ$-basis.

As $\mathcal U = \{ u_k : k \in \bL \}$ spans $\bR^n$, it contains a
basis, and the conditions of Proposition \ref{finite} imply this is
a $\bZ$-basis. We may change basis of $\bZ^n$ so that this
set is $\{ \mathbf e_1, \ldots, \mathbf e_n \}$. Now for each $u_k$ we may
write $u_k = \sum_{i=1}^{n} a_{ki} \mathbf e_i$, and for each $i$ with
$a_{ki} \neq 0$, then $\{ \mathbf e_j : j \neq i \} \cup \{ u_k \}$
is a $\bZ$-basis. Thus the matrix with these vectors as columns is
in $GL(n, \bZ)$ so has determinant $\pm 1$. But up to sign this determinant is
$a_{ki}$ so in general the coefficients of $u_k$ lie in
$\{ -1,0,1 \}$ and there are at most $3^n - 1$ distinct elements in
$\mathcal U$, thus establishing Proposition \ref{finite}.

It follows that under condition~\itref{item:b}, the set
\( \Set{\norm{u_k} \with k\in\bL} \) is bounded.
In particular it is now easy to
show that  \( \beta\colon \lie t_\lambda \to \bR^n \) 
is indeed continuous and it
follows that \( N_\beta \) is a Hilbert subgroup of~\( T_\lambda \) of
codimension~\( n \).

\begin{remark} \label{Goto} 
When $\bL$ is finite, the convergence condition (\ref{eq:convergence})
is vacuous, $M_{\Lambda} = \bH^\bL$, and we have the finite-dimensional
construction of \cite{BD}.
The construction of Hattori
  \cite{Hattori:Ainfty-volume} corresponds to \( n = 1 \), \(\bL = \bZ \) and
  \( u_k = 1 \in \bR \) for each~\( k \).  For general
  dimension~\( 4n \), Goto's construction \cite{Goto:A-infinity}
  corresponds to
  \( \bL = \bN_{>0} \mathbf i \cup \bN_{>0} \mathbf k \cup \Set{1,\dots,n} \),
  \begin{equation*}
    \Lambda_k =
    \begin{dcases*}
      k,&for \( k \in \bN_{>0} \mathbf i \),\\
      -k,&for \( k \in \bN_{>0} \mathbf k \),\\
      0,&for \( k \in \Set{1,\dots,n} \),
    \end{dcases*}
    \quad\text{with}\quad
    u_k =
    \begin{dcases*}
      \mathbf e_1,&for \( k \in  \bN_{>0} {\bf i} \cup \bN_{>0} {\bf k} \),\\
      \sum_{i=1}^n \mathbf e_i,&for \( k = 1 \),\\
      - \mathbf e_r,&for \( k \in \Set{2,\dots,n} \).
    \end{dcases*}
  \end{equation*}
  Thus Goto's construction is for one concrete choice of
  \( (\lambda_k)_{k\in\bL} \) and only one of the \( u_k \)'s is
  repeatedly infinitely many times.
\end{remark}

\medskip 
Just as in Hattori \cite{Hattori:Ainfty-volume}, one may use
the \( T_\lambda \)
action to show that different choices of \( (\Lambda_k)_{k\in\bL} \)
yielding the same \( (\lambda_k)_{k\in\bL} \)
result in hyperKähler structures that are isometric via a
$T^n$-equivariant tri-holomorphic map. Choosing a complex structure and writing
$\Lambda=(\Lambda_z, \Lambda_w)$ we shall often assume that either
$\Lambda_z =0$ or its terms grow like $O(k)$, and similarly for
$\Lambda_w$.

The full classification of hypertoric manifolds relies on the following 
operation from \cite{Bielawski:tri-Hamiltonian} that does not change the
topology: given a hyperK\"ahler manifold $M^{4n}$ with tri-Hamiltonian
$T^n$ action, a {\em Taub-NUT deformation} is any smooth hyperK\"ahler
quotient of $M \times \bH^m$ by $\bR^m$ where $\bR^m \subset \bH^m$
acts by translation on $\bH^m$ and via an injective linear map
$\bR^m \rightarrow {\rm Lie}(T^n)$ on $M$.

\begin{theorem} \label{classification} \cite{DS-SMS}
Let $M$ be a connected complete hypertoric manifold of dimension $4n$. Then
$M$ is topologically a product $M = M_2 \times (S^1 \times \bR^3)^m$ with $M_2 = M(\beta, \lambda)$ for some $\beta, \lambda$. The hyperK\"ahler
metric on $M$ is either the product hyperK\"ahler metric or a Taub-NUT
deformation of this metric.
\end{theorem}

\begin{remark}
  Several authors (see \cite{BLPW} for example) have explored the notion of Gale duality
  for hypertoric varieties. This swaps $\n = {\rm Lie} (N)$ and
  $\n^\perp$, its orthogonal in $\tf = {\rm Lie} (T)$. Now $\n^\perp$ can
  be identified with the Lie algebra of the isometry group $T/N$ of the
  hyperk\"ahler quotient.  The duality also interchanges the level $\eta$ at
  which we reduce and an element $\xi$ of the Lie algebra of the
  isometry group. Thus deformation and isometry parameters are interchanged.

  This duality is an example of the notion of `symplectic duality',
  which ultimately comes from physics (duality between Coulomb and 
Higgs branches). The
  swapping of deformation and isometry parameters is a general feature
  of this picture.

  In our set-up, $\n$ has finite codimension in $\tf_{\lambda}$ so the
  'Gale dual' would be a quotient of $M_{\Lambda}$ by a finite-dimensional subgroup of
  $T_\lambda$, and hence would be an infinite-dimensional hypertoric.
\end{remark}

\section{The $T^n$ action on the hypertoric manifolds}
\label{sec:properties}

We have defined $M$ as a hyperk\"ahler quotient of $M_\Lambda$ by the
finite codimension subtorus $N_\beta$ of the Hilbert group $T_\lambda$.
In this section we consider further the hypertoric structure of $M$.

We first recall that \( T^n = T_\lambda/N_\beta \) acts on~\( M \) 
preserving the induced hyperKähler
structure, and with moment
map~\( \phi\colon M \to \iH \otimes (\bR^n)^* \) induced
by~\( \mu_\Lambda \).  In more detail,
\( \nLie T^n = \lie t_\lambda/\lie n_\beta \) and hence
\( (\nLie T^n)^* = (\lie n_\beta)^0 \), the annihilator of
\( \lie n_\beta \) in~\( \lie t_\lambda^* \). (We recall here that under our hypotheses \( \lie n_\beta \) is a closed subspace of \( \lie t_\lambda \) ). Now
on~\( \mu_\beta^{-1}(0) \) the map~\( \mu_\Lambda \) takes values in
\( (\lie n_\beta)^0 = (\nLie T^n)^* \) and descends to~\( M \)
as~\( \phi \).

Explicitly, $\phi(x) =a $ where
$a \in \iH \otimes (\bR^n)^{*} = \iH \otimes ({\rm Lie} (T^n))^{*}$ is
defined by (\ref{eq:im-beta}), that is
\[
a(u_k) = \lambda_k + \tfrac12\overline x_k\mathbf ix_k \;\; : \;\; k \in \bL
\]
 Note that $a$ is uniquely determined by
$x$ if we assume, as we always do, that the $u_k$ generate
\( \bR^n \).

In particular, observe that the stabiliser in $T_{\lambda}$ of $x$ has
Lie algebra spanned by the vectors $e_k$ where $x_k=0$, or
equivalently, using (\ref{eq:im-beta}), (\ref{eq:flat}), the set of
vectors $e_k$ where $a = \phi(x)$ is in the flat $H_k$ defined
by $a(u_k) = \lambda_k$. The analogous statement
for the $T^n$-stabiliser is that it is spanned by the $u_k$ such that
$a = \phi(x) \in H_k$.

Note also that the Lie algebra of $T^n$ may be identified with the
span of a finite collection of $e_k$ (we can take $e_k : k \in I_1$
where $u_k : k \in I_1$ is a basis for $\R^n$).

The above discussion implies, as in~\cite{BD}:

\begin{lemma} \label{phifibre} The hyperk\"ahler moment map $\phi$ for
  the $T^n$-action on $M$ induces a homeomorphism
  \( M/T^n \to \iH\otimes(\bR^n)^* = \bR^{3n}\). In particular, $M$ is connected.

  For \( p \in M \), the stabiliser \( \stab_{T^n}(p) \) is the
  subtorus with Lie algebra spanned by the \( u_k \) such that
  \( \phi(p) \in H_k \).
\end{lemma}

\bigskip
We saw in the discussion after Proposition \ref{finite} that if a collection
of flats $H_k$ intersect then the corresponding $u_k$ are contained in
a $\bZ$-basis. Combining this with Lemma \ref{phifibre}, we see, as in
the finite case:

\begin{proposition} \label{fibres2} Under the hypotheses of Theorem
  \ref{thm:construction}, we have that if $a \in \iH\otimes(\bR^n)^* \cong
 \bR^{3n}$
  lies in exactly $r$ flats then the $T^n$-stabiliser of a point in
  $\phi^{-1}(a)$ is an $r$-dimensional torus.
\end{proposition}

From this result and Lemma \ref{phifibre}, we see $\phi$ maps $M$ onto
$\bR^{3n}$ with generic fibre $T^n$, but these fibres may collapse to
lower dimensional tori depending on how many flats contain the point
in $\bR^{3n}$. In particular the intersections of $n$ flats correspond
to fixed points of the $T^n$ action on $M$.

The pre-image of the flat $H_k$ is, from the discussion earlier, the
hyperk\"ahler subvariety $S_k$ of $M$ given by the 
quaternionic condition $x_k=0$. So we have a
collection $S_k (k \in \bL)$ of $T^n$-invariant hyperk\"ahler
subvarieties of $M$, such that the $T^n$ action is free on the
complement of $\bigcup_{k \in \bL} S_k$ in $M$. The moment map $\phi$
expresses $M - \bigcup_{k \in \bL} S_k$ as a $T^n$-bundle over
$\bR^{3n} - \bigcup_{k \in \bL} H_k$ . The real dimension of each
$S_k$ is $4(n-1)$.

\medskip
 We can also split the $T^n$-moment map into real and complex
parts; equivalently, we write $a \in \iH \otimes (\bR^n)^{*}$ as
$(a_{\bR}, a_{\bC})$, which we view as lying in $(\bR^n)^* \oplus
(\bC \otimes (\bR^n)^*)$. This amounts to choosing a particular complex structure, which
we refer to as $\sf I$, in the 2-sphere of complex structures defined by
the hyperk\"ahler structure. So we split $\bH = \bC \oplus \mathbf j 
\bC$ and write $x \in \bH^\bL$ as $x = z + \mathbf j w$ with $z,w \in \bC^\bL$.

Let us write $\lambda_k \in \iH$ as \( (\lambda_k^1, \lambda_k^2+ i 
\lambda_k^3) = (\lambda_k^\bR, \lambda_k^{\C}) \in \bR \oplus \bC \). 
Equation~\eqref{eq:im-beta} can now be written as
\begin{equation} \label{betareal} a_{\bR}(u_k)
  =\lambda_k^\bR + \frac{1}{2} (|z_k|^2 - |w_k|^2)
\end{equation}

\begin{equation} \label{betacomplex} a_{\bC}(u_k)  =
  \lambda_k^{\C} + z_k w_k
\end{equation}
(This was the form of the equations used in
\cite{BD}).

It is useful to look at just the complex moment map
$\phi_{\bC} : (z,w) \mapsto a_{\bC}$. A fibre of $\phi_{\bC}$ will map
via $\phi_{\bR}$ onto $(\bR^n)^*  \cong \bR^n$ with generic fiber a copy of $T^n$.  If
$a_{\bC}$ is in the complement of the complex flats
\begin{equation} \label{cplexflats}
  H_{k, \bC} = \{ b \in \bC^n \otimes (\bR^n)^* : b(u_k) = \lambda_k^\C  \}
\end{equation}
then $(a_{\bR}, a_{\bC})$ is in the complement of the $H_k$ for all
$a_{\bR}$, and hence $\phi_{\bC}^{-1} (a_{\bC})$ is just
$T^n \times \bR^n$.  If $a_{\bC}$ does lie in some $H_{k, \bC}$, then
$\phi_{\bR} : \phi_{\bC}^{-1} (a_{\bf C}) \rightarrow \bR^n$ has
degenerate fibres (tori of dimension less than $n$) over points
$a_{\bR}$ in the hyperplanes
$ a_{\bR}(u_k) = \lambda_k^\bR$.

Note that  $\phi_{\bC}^{-1} (a_{\bC})$ is a toric variety for the $T^n$
action (the moment map is projection onto the $\bR^n$ factor).

\begin{remark}
  In \cite{BD} the authors considered the case
  when all $\lambda_k$ are set to zero. The space $M(\beta, \lambda)$
  is now a cone with action of $\R^{+}$. It was shown that the origin
  was the only singularity (and hence $M(\beta, \lambda)$ was a cone
  over a smooth 3-Sasakian manifold) if and only if every set of $n$
  vectors $u_k$ is linearly independent and every set of less than $n$
  vectors $u_k$ is part of a $\bZ$-basis. In our situation these
  conditions are never satisfied because Lemma \ref{finite} shows that
  some $u_i$ are necessarily repeated infinitely often. This is to be expected
as of course the 3-Sasakian space, being compact, cannot have infinite topological type.
\end{remark}

\section{Complex structures}
In this section we consider complex structures on the hypertoric quotients.
For ease of notation we shall suppress the subscript $\beta$ and just write
$N$ etc.

\smallskip
The complexification \(T_\lambda^{\bC} \)
of  \( T_\lambda \) is the group
\begin{equation*}
  T_\lambda^{\bC}  = \Set[\Big]{ g \in
  (\bC^*)^{\bL} \with \sum_{k\in\bL} (1 +
  \abs{\lambda_k})\, \abs{1-g_k}^2 < \infty}.
\end{equation*}
This acts on $M_{\Lambda}$ and induces an action of the complexification 
 $N^{\C}$ of $N$.
The zero locus of the complex part of our moment map for $N$ consists of
the pairs $(z,w) \in M_{\Lambda}$ satisfying equation (\ref{betacomplex}) for some
$a_{\bC} \in \bC^n$. This locus can be identified with the set of $(a_{\bC}, (z,w)) \in \bC^n \times M_{\Lambda}$ satisfying equation (\ref{betacomplex}),
as $a_{\bC}$ is determined by $(z,w)$ since the $u_k$ span $\bC^n$ over $\bC$.

We recall the definition (\ref{cplexflats}) of the complex flats $H_{k, \bC}$.

\begin{proposition} \label{discrete}
Suppose that each set of $n+1$ flats $H_{k, \bC}$ has empty intersection. Then
the action of $N^{\bC}$  on 

\begin{equation} \label{cplexNmap}
\{ (a_{\bC}, (z,w)) :   a_{\bC}(u_k)  = \lambda_k^{\C}  + z_k w_k \}
\end{equation}
by $t. (a_{\bC}, (z,w)) = (a_{\bC}, (t.z, t^{-1}.w))$, has discrete stabilisers.
\end{proposition}

\begin{proof}
The stabiliser of $(z,w)$ for the  $N^{\bC}$ action is contained in 
 $(T_I)^{\bC}$ , where $I$ is the set
of indices $k$ where $z_k w_k =0$, or equivalently the set of indices 
where $a_{\C} (u_k) = \lambda_k^\C$, ie. $a_{\C}  \in H_{k , \bC}$.
The argument after Proposition \ref{finite} shows that if a collection of
complex flats $\{ H_{k, \bC} : k \in I \}$ intersect, then the
$ \{u_k : k \in I  \}$ are linearly independent, so $\beta : e_i \mapsto u_i$ 
is injective on Lie $(T_I)^\C$, and hence $\n^{\C} \cap {\rm Lie \;} (T_I)^{\C} = 0$. We thus have that
stabilisers  $N^{\bC}$ are discrete.
\end{proof}

Observe also that the argument of Thm 5.1 of \cite{BD}
shows that, subject to the assumptions of Proposition \ref{discrete},
 $\{ \alpha_i : i \in I^c \}$ spans $\lie n^*$, where 
$I^c$ denotes the complement of $I$. The key point
is to show that $\n \cap {\rm Lie\;} (T_I) = 0$ implies 
$\lie t = \n^{\perp} + {\rm Lie\;} (T_I)$. The space $\lie t$ is now an infinite-dimensional Hilbert space but we still have $(X \cap Y)^{\perp} = 
\overline{X^{\perp} + Y^{\perp}}$ for {\em closed} subspaces $X$ and $Y$. Now  $\n$ is closed, as
remarked earlier, and Lie$(T_I)$ is closed as $I$ is finite, so
we have $\lie t = \overline{\n^{\perp} +  {{\rm Lie\;} (T_I)}^\perp}$.
But also $\n^{\perp}$ is finite-dimensional (as $\n$ has finite codimension) 
so $\n^{\perp} +  {{\rm Lie\;} (T_I)}^\perp$ is closed, since the sum
of a closed and a finite-dimensional space is closed.

\medskip
Our next aim is to identify, under certain conditions, the complex quotient $(\mu^{\C})^{-1}(0)/N^{\C}$
with the hyperk\"ahler quotient $\mu^{-1}(\lambda^\bR, 0)/N$.

We shall assume that we are in the situation of Proposition
\ref{discrete}.  In addition, we need to make certain assumptions
about the growth rates of the components of $\Lambda$.  Choosing a
complex structure, let us write $\Lambda = (\Lambda_z, \Lambda_w)$.
Now $|\lambda_k| = |\Lambda_k|^2 = |(\Lambda_z)_k|^2 + |(\Lambda_w)_k|^2$. The
$\lambda_k$ need to grow fast enough for condition (\ref{eq:convergence})
to hold, and we shall assume that they grow like $O(k^\delta)$ for
some $\delta > 1$.  (In concrete examples, one often takes
$\delta = 2$).

We shall assume in the following
discussion that the growth rates of $(\Lambda_z)_k$ and $(\Lambda_w)_k$
are {\em commensurable} in the sense that there exist positive constants
$C_1$ and $C_2$ independent of $k$ such that
\[
C_1 |(\Lambda_w)_k| \leqslant |(\Lambda_z)_k| \leqslant C_2 |(\Lambda_w)_k|
\]
for sufficiently large
 $k$. This implies of course that both $|(\Lambda_z)_k|^2$ and 
$|(\Lambda_w)_k|^2$ have commensurable growth rates with $\lambda_k$, and
in particular tend to infinity as $k \rightarrow \infty$.

Now a point $(z,w) \in M_{\Lambda}$ may be written
$(z,w) = (\hat{z}, \hat{w}) + (\Lambda_z, \Lambda_w)$ where
$\hat{z}$ and $\hat{w}$ are square-summable. So both
$|z_k|^2$ and $|w_k|^2$ have commensurable growth rates with $\lambda_k$
also.

\medskip
We shall now turn to the study of the quotient $(\mu^{\C})^{-1}(0)/N^{\C}$.
We shall make use of the theory of monotone operators, as developed
for example in Showalter's book \cite{Showalter}.

\medskip
(i) Let us consider the $N^{\bC}$-orbit $\mathcal O$ 
through a point $(z,w) \in \mu_{\C}^{-1}(0)$
 as above. We want to show that the real moment map $\mu^\R$ is surjective
when restricted to the orbit.

\smallskip
As $N$ is abelian, we have $\lie n^\bC = \lie n + \lie a$ with $\lie a
= i \lie n \cong \lie n$ Abelian. Writing $A = \exp(\lie a)$ we have
 $N^\C = NA$. Now  $\mu^\bR$ is $N$-invariant, as it is the moment
map for an Abelian action,
so we can view $\mu^\bR$, restricted to the orbit $\mathcal O$, 
as a map from $\lie a$ to $\lie n^{*}$
(identifying $A$ with $\lie a$ via the exponential map).
Explicitly
\[
\mu^\bR (e^y.(z,w))= \frac{1}{2} \sum_{i \in \bL} \left( |z_i|^2 e^{2 \alpha_i . y} 
- |w_i|^2 e^{-2 \alpha_i . y} \right) \alpha_i + c_1
\]
where $y \in \lie a$. (Compare
with the expressions in \cite{BD} and 
\cite{Guillemin}).
Also $\alpha_i = \iota^* e_i^* \in \lie n^{*}$
and $c_1 = \sum \lambda_i^1 \alpha_i$ is the constant term.

Writing  $\tilde{\mu}(y)= \mu^\bR (e^y.(z,w))$, we get a map
$\tilde{\mu} : \lie n \rightarrow \lie n^*$,
with the same image as $\mu^{\bR}$ on the orbit:

\[
\tilde{\mu}(y) = \frac{1}{2} \sum_{i \in \bL} \left( |z_i|^2 e^{2 \alpha_i . y} 
- |w_i|^2 e^{-2 \alpha_i . y} + 2 \lambda_i^1\right) \alpha_i 
\]
This is the Legendre transform $D \sf F$
of the map ${\sf F} : \lie n \rightarrow \R$
\[
{\sf F} = \frac{1}{4} \sum_{i \in \bL} \left( |z_i|^2 e^{2 \alpha_i . y} 
+ |w_i|^2 e^{-2 \alpha_i . y}  - (|z_i|^2 + |w_i|^2) + 4 \lambda_i^1 \alpha_i.y 
\right)  
\]
The constant term $(|z_i|^2 + |w_i|^2)$ is subtracted to ensure convergence.
In more detail, let us
write $t_i = 2 \alpha_i.y$, so for $y \in \lie n$ we have
$t_i = 2 e_i^*(y) = 2y_i$.
We now see that $\sf F$ is $\frac{1}{4}$ times a sum of terms
\[
(|z_i|^2 - |w_i|^2 + 2 \lambda_i^1) t_i +
|z_i|^2 (e^{t_i} - 1 - t_i) +
|w_i|^2 (e^{-t_i} - 1 + t_i).
\]
The sum of the first terms just gives $\tilde{\mu}(0)=\mu^\bR(z,w)$ evaluated on $y \in \lie n$.
The second and third terms are $\sim |z_i|^2 t_i^2$ (resp. $\sim |w_i|^2 t_i^2)$)
for large $i$, so the sums converge, since 
$\sum (1 + |\lambda_i|) t_i^2 < \infty$ by our definition of the Lie algebra
$\tf_{\lambda}$ of which $\lie n$ is a subalgebra.

\smallskip
Now $D^2 {\sf F} = \sum_{i \in \bL} \left( |z_i|^2 e^{2 \alpha_i . y}+ |w_i|^2 e^{-2 \alpha_i . y} \right)
\alpha_i \otimes \alpha_i$. From above, $\{\alpha_i : i \in I^c \}$ spans
$\n^*$ where $I^c$ is the set of indices $i$ for which  $z_i$ and $w_i$ are both nonzero, so $\sf F$ is convex, and $\tilde{\mu} = D \sf F$
is a \emph{monotone operator}, in the sense that
$\langle \tilde{\mu}(u) - \tilde{\mu}(v), u-v \rangle \geq 0$ for all $u,v$.
(This follows by applying the one-variable mean value theorem to
the function $t \mapsto \langle \tilde{\mu}(t u + (1-t)v), u-v \rangle$).

\smallskip
(ii) To show surjectivity of the
monotone operator $\tilde{\mu} : \n \rightarrow \n^*$ it remains 
(see \cite{Showalter} section II.2) to show
that $\tilde{\mu}$ maps bounded sets to bounded sets, and
 to prove a coercivity condition 
\[
\frac{\tilde{\mu}(y).y}{|| y ||} \rightarrow \infty \;\; {\rm as} \;\; ||y|| \rightarrow \infty.
\]
We first show coercivity.
 Writing as above $y_i =  \alpha_i . y$, so
\[
\tilde{\mu}(y).y =  \frac{1}{2} \sum_{i \in \bL} \left( |z_i|^2 e^{2y_i} 
- |w_i|^2 e^{-2y_i}  + 2 \lambda_i^1 \right) y_i
\]

Let us look at the $i$th term, which we can rewrite (suppressing the
overall factor of $\frac{1}{2}$) as
\begin{equation} \label{ti}
y_i ( |z_i|^2 - |w_i|^2 + 2 \lambda_i^1) + f(y_i) |z_i|^2 + f(-y_i) |w_i|^2
\end{equation}
where $f(t) = t (e^{2t} - 1) = 2 t \sinh(t) e^t$.
Note $f(t)$ is always nonnegative and we have $f(0)=f^\prime(0)=0$
and $f^{\prime \prime}(t) = 4(1+t) e^{2t}$, so $f(t) \geqslant 2t^2$ for
all $t \geqslant 0$.

Moreover the first term in (\ref{ti})
is just the $i$th term in the linear fnctional $\tilde{\mu}(0)$ 
evaluated on $y$. Thus
\[
2 \tilde{\mu}(y) (y) \geqslant \tilde{\mu}(0)(y) + \sum_{i \in \bL}
2 \min(|z_i|^2, |w_i|^2) y_i^2. 
\]


Recall $(z,w) = (\hat{z}, \hat{w}) + (\Lambda_z, \Lambda_w)$ where
$\hat{z}$ and $\hat{w}$ are square-summable.
We see there exists a constant $B_z$ depending on $z$ but independent of
$i$ such that $|z_i|^2 \leqslant B_z (1 + |\lambda_i|)$ for all $i$, and
similarly for $w_i$.
Moreover the discussion on comensurability above shows that
there exists a {\em positive} constant
$D_z$ such that $|z_i|^2 \geqslant D_z (1 + |\lambda_i|)$ for all $i$ such that
$z_i \neq 0$. Again we have an analogous statement for $w$.

Recall now that $z_i w_i =0$ if and only if $i \in I$.
We deduce that, subject to our assumptions, 
$\min(|z_i|^2, |w_i|^2)$ is bounded below by a positive constant 
times $(1 + |\lambda_i|)$ for $i \in I^c$.

Now $\tilde{\mu}(y)(y)$ is the sum of an expression that grows at most linearly, 
and a term bounded below by a positive constant times
$\sum_{i \in I^c} (1 + |\lambda_i|) t_i^2$, which is greater than or equal to a positive constant 
times $|| y ||^2$ as, from above, the $\alpha_i$ span $\n^*$.

\medskip
(iii) To show $\tilde{\mu}$ maps bounded sets to bounded sets, we need to check that 
$\tilde{\mu}(y) (\xi)$ can be bounded in terms of $y$ and linearly in $|| \xi ||$
for $y, \xi \in \n$.
A similar calculation to that above gives that
\begin{equation} \label{muyxi}
\tilde{\mu}(y)(\xi) = \tilde{\mu}(0) (\xi) + 
\frac{1}{2}\sum_{i \in \bL} (|z_i|^2 (e^{2y_i}-1) - |w_i|^2 (e^{-2y_i}-1)) 
\alpha_i(\xi)
\end{equation}
The first term is bounded linearly in terms of $|| \xi ||$.

A bound on $y$ gives a bound, uniform in $i$, on the $y_i$ (since
$|\lambda_i| \rightarrow \infty$ as $i \rightarrow \infty$).
Writing $e^t -1 = t h(t)$ where $h(t)= e^t \sinh(t)/t$ is continuous, we see that, for bounded
$|| y ||$, we have
$|z_i|^2|e^{2y_i}-1| \leqslant C |z_i|^2 |y_i|$ for a constant $C$ independent of $i$,
and we similarly bound $|w_i|^2(e^{-2y_i}-1)$.

Now $\sum_{i \in \bL} |z_i|^2 (e^{2y_i}-1) \alpha_i(\xi)$
is bounded by 
\[
C \sum  |z_i|^2 |y_i| |\alpha_i(\xi)| \leq
C \left( \sum |z_i|^2 y_i^2 \right)^{\frac{1}{2}} 
\left( \sum |z_i|^2 (\alpha_i(\xi))^2 \right)^{\frac{1}{2}}
\]
As we have a constant $B$ independent of $i$ such that
$|z_i|^2, |w_i|^2 \leqslant B (1 + |\lambda_i|)$, we see finally the second and third
terms in (\ref{muyxi}) are bounded by a constant depending on $(z,w)$
and the bound on  $|| y ||$, times $|| \xi ||$, thus concluding the
argument.

\medskip
(iv) So, subject to our hypotheses, we
have that the $N^\C$-orbit through $(z,w)$ meets each level set
of $\mu^{\bR}$. Standard arguments (see eg Prop 5.1 of \cite{Goto:A-infinity})
show that in fact the $N^\C$-orbit meets the the level set in an $N$-orbit.

So, we can identify the
hyperk\"ahler quotient with the $N^\C$-quotient of the zero locus of the
 complex equation (\ref{betacomplex}).

\medskip
We now deduce that $M(\lambda^1, \lambda^2, \lambda^3)$ is diffeomorphic to
$M(\lambda^1, \lambda^2, 0)$ and (rotating complex structures) with
$M(\lambda^1, 0, 0)$. We shall analyse the topology of this space in the
next section.

\begin{remark}
  Considering $M(\lambda^1, 0, 0)$) means we take $\Lambda_z$ or
  $\Lambda_w = 0$, so the $\lambda_k$ are pure multiples of $i$. 
  We now have a holomorphic (but not triholomorphic)
  circle action on the hyperkahler quotient, induced by rotation in
  the $z$ (resp. $w$) factor. For, as this rotation fixes $\Lambda$ it
  induces an action on $M_\Lambda = \Lambda + L^2 (\bH)$ and hence
  on the hyperk\"ahler quotient, since
  $ \lambda_k^\C = \lambda_k^2 + i\lambda_k^3=0$.  In terms of
  equations (\ref{betareal}) and (\ref{betacomplex}), $a_{\bR}$
  remains unchanged while $a_{\bC}$ is rotated. The fixed point set of
  this action is therefore contained in $\phi_{\bC}^{-1}(0)$, and is a
  union of toric varieties contained in the locus $z_k w_k =0$.
\end{remark}

\section{Topology}

We can analyse the topology using similar methods to the finite case.
In view of the above discussion,
let us take $\lambda_k = (\lambda_k^\bR,0,0)$, that is, to have zero
complex part.

First recall the homeomorphism
$\tau : \bR_{\geq 0} \times \bR_{\geq 0} \rightarrow \R \times
\bR_{\geq 0}$ defined by
\begin{equation} \label{tau} (x,y) \mapsto (\frac{1}{2}(x^2-y^2), xy)
\end{equation}
whose inverse is
\[
  \tau^{-1} : (p,q) \mapsto (+ \sqrt{p + \sqrt{p^2 + q^2}}, + \sqrt{
  -p + \sqrt{p^2 + q^2}})
\]
The deformation $j_t : (p,q) \mapsto (p,tq)$ of
$\bR \times \bR_{\geq 0}$ now induces a deformation
$\jmath_t = \tau^{-1} \circ j_t \circ \tau$ of
$\bR_{\geq 0} \times \bR_{\geq 0}$. We shall take the deformation
parameter $t$ to lie in $[0,1]$. 

Explicitly, 
\begin{equation} \label{j1}
\jmath_t^{1}(x,y) = \sqrt{ \frac{1}{2}(x^2-y^2) + \frac{1}{2}
\sqrt{(x^2 - y^2)^2 + 4t^2 x^2 y^2}}
\end{equation}
\begin{equation} \label{j2}
\jmath_t^{2}(x,y) = \sqrt{ \frac{1}{2}(y^2-x^2) + \frac{1}{2}
\sqrt{(x^2 - y^2)^2 + 4t^2 x^2 y^2}}
\end{equation}
Note that $\jmath_1$ is the identity,
while $\jmath_0$ sends $\bR_{\geq 0} \times \bR_{\geq 0}$ into the
union of the non-negative $x$ and $y$ axes. 

Writing
$\jmath_t(x,y) = (\jmath_t^{1}(x,y), \jmath_t^{2}(x,y))$, we obtain a
deformation $h_t$ of $\bH = \bC \times \bC$ given by
\[
  h_t (z,w) = \left( \jmath_t^{1}(|z|,|w|)\frac{z}{|z|},
    \jmath_t^{2}(|z|,|w|)\frac{w}{|w|} \right)
\]
This can now be extended diagonally to a $T^{\bL}$-equivariant
deformation $h_t$ of $\bH^{\bL}$ with $h_1$ equal to the identity.

Note that
$(\jmath_t^{1}(x,y))^2 + (\jmath_t^{2}(x,y))^2 \leq x^2 + y^2$ for
$0 \leqslant t \leqslant 1$ so this actually induces a deformation of
$\bL^2 (\bH)$. Moroever, as $(\jmath_t^{1}(x,y))^2 - (\jmath_t^{2}(x,y))^2 
= x^2 - y^2$, we deduce that $(\jmath_t^{1}(x,y))^2 \leqslant x^2$ and
$(\jmath_t^{2}(x,y))^2 \leqslant y^2$.

Now $| h_t(z,w) - (z,w) |^2 = (\jmath_t^{1}(x,y) - x)^ 2 +
(\jmath_t^{2}(x,y)-y)^2$, where $x = |z|$ and $y= |w|$.

The second term is bounded by $4y^2$. The first term is $O(1/x^2)$
for $x$ large in comparison with $y$.

In our situation (as we are taking $\lambda_k = (\lambda_k^\bR,0,0)$,
we can take $\Lambda = (\Lambda^{(1)}, 0)$. Now let us consider
$(z_k,w_k)_{k\ \in \bL} \in M_{\Lambda}$, so $(z_k,w_k) = (\Lambda_k + \hat{z}_k, w_k)$
where $(\hat{z}_k)$ and $(w_k)$ are square summable sequences. 
Moroever $|z_k|^2 = O(\lambda_k)$, so, 
since $\sum_{k} \frac{1}{\lambda_k} < \infty$
the above calculations show $(h_t(z_k,w_k) - (z_k,w_k))$ is square-summable.
So if $(z,w) - \Lambda \in L^2(\bH)$ then $h_t(z,w) - \Lambda \in L^2(\bH)$
also.

So $h_t$ in fact gives a deformation of $M_{\Lambda} = \Lambda + L^2(\bH)$. 
The equivariance implies that
this induces a deformation of the hyperk\"ahler quotient
$M = M(\beta, \lambda)$.

As
\[
  \tau \circ \jmath_t(x,y) = j_t \circ \tau(x,y) =
  (\frac{1}{2}(x^2-y^2), txy)
\]
we see that the real component $\mu^\bR$ of our moment map is invariant
under the deformation, while the complex component $\mu^\bC =\mu_2 + i \mu_3$
scales by $t$.  Passing to the hyperk\"ahler quotient $M$, similar
statements apply to the moment map $\phi$ for the $T^n$-action on $M$.

Now the family $h_t \; (0 \leq t \leq 1)$ deforms $M$ to
$\phi_{\bC}^{-1}(0)$, i.e. the locus where $a_{\bC}=0$.  
As $\lambda_k^\bC =0$, points
$(z,w)$ of this locus satisfy
\[
  z_k w_k =0.
\]
Moreover, from the moment map equation (\ref{betareal}), we have
\[
   a_{\bR}(u_k)  =\lambda_k^\bR + \frac{1}{2} (|z_k|^2 -
  |w_k|^2)
\]
So if $a_{\bR}$ is on the positive side of the hyperplane
$H_{k, \bR} = \{ x :  x(u_k)  = \lambda_k^\bR \}$ in
$\bR^n$, then $w_k =0$ and
$|z_k|^2 = 2(a_{\bR}(u_k)  - \lambda_k^\bR)$.  If $a_{\bR}$ is on
the negative side of the hyperplane then $z_k=0$ and
$|w_k|^2 = -2(a_{\bR}( u_k)  - \lambda_k^\bR)$.  So
$\phi_{\bC}^{-1}(0)$ is a union of the (in general compact and
non-compact) toric varieties corresponding to the (in general bounded
and unbounded) polyhedra in $\bR^n$ defined by the arrangement of
hyperplanes $H_{k, \bR}$.  Now as in \cite{BD} we
can retract $\phi_{\bC}^{-1}(0)$ onto the union of compact toric
varieties corresponding to the bounded polytopes.

In summary, the hypertoric variety has the homotopy type of a union of
(in general infinitely many) compact toric varieties intersecting in
toric subvarieties.

These toric varieties are K\"ahler submanifolds of $M$ with respect to
the complex structure $\sf I$, because the condition $\phi_{\bC}=0$
(i.e.\ $z_k w_k=0$ for all $k$ ) is $\sf I$-holomorphic.  As the
$\sf I$-holomorphic complex symplectic form
$\omega_{\bC}=\omega_{\sf J} + i \omega_{\sf K}$ is induced from the form
$dz \wedge dw$ on $\bH$, we also see that these toric varieties are
complex Lagrangian with respect to $\omega_{\bC}$.

\begin{example}
  Let us revisit the example due to Goto \cite{Goto}.  Now the array of
  hyperplanes is obtained from that defining a simplex by adding
  infinitely many translates of the hyperplane defining one face. We
  obtain a picture where the bounded chambers are a pair of simplices
  and an infinite collection of simplices truncated at one vertex. The
  corresponding compact toric varieties are a pair of complex
  projective spaces $\bC \bP^n$ and an infinite collection of the
  blow-ups of $\bC \bP^n$ at one point. Each toric variety intersects
  the next one in a $\bC \bP^{n-1}$ (corresponding to the
  $(n-1)$-simplex where the associated chambers meet), except that the
  two projective spaces meet in a point, because the associated
  chambers just meet in a point.
  We illustrate this for $n=2$ in Figure \ref{fig:inf-config}.
  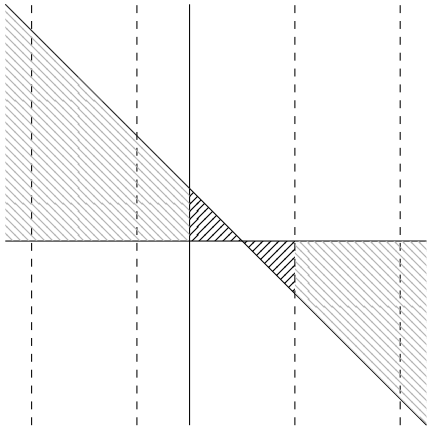
\begin{figure}
    \centering
    \begin{tikzpicture}[scale=.7]
      \newcommand{\ca}{-3.5}
      \newcommand{\cb}{4.5}
      \draw (\ca,0)--(\cb,0);
      \draw (0,\ca)--(0,\cb);
      \draw (\ca,\cb)--(\cb,\ca);
      \foreach \x in {-3,-1,2,4} {
      \draw[dashed] (\x,\ca)--(\x,\cb); };
      \fill[pattern=north east lines] (0,0) -- (0,1) -- (1,0) --
      cycle;
      \fill[pattern=north east lines] (1,0) -- (2,0) -- (2,-1) --
      cycle;
      \fill[pattern=north west lines,pattern color=gray!60] (\ca,0)
      -- (0,0) -- (0,1) -- (\ca,\cb) -- cycle;
      \fill[pattern=north west lines,pattern color=gray!60] (2,0) --
      (\cb,0) -- (\cb,\ca) -- (2,-1) -- cycle;
    \end{tikzpicture}
    \caption{Singly infinite configuration in two dimensions}
    \label{fig:inf-config}
  \end{figure}

  In the two-dimensional case, we can also have doubly or triply
  infinite configurations, see Figure \ref{fig:dt-inf}.
  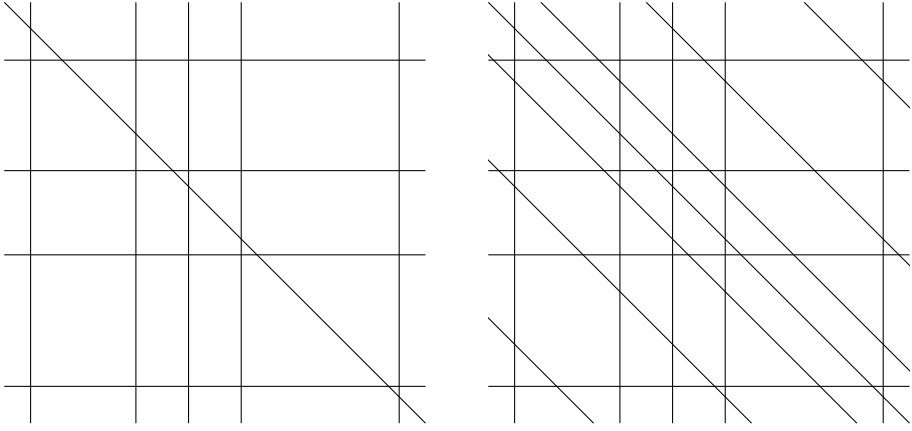
\begin{figure}
    \centering
    \begin{tikzpicture}[scale=.7]
      \newcommand{\ca}{-3.5}
      \newcommand{\cb}{4.5}
      \draw (\ca,\cb)--(\cb,\ca);
      \foreach \x in {-3,-1,0,1,4} {
      \draw (\x,\ca)--(\x,\cb); };
      \foreach \y in {-2.8,-0.3, 1.3, 3.4} {
      \draw (\ca,\y)--(\cb,\y); };
    \end{tikzpicture}
    \qquad
    \begin{tikzpicture}[scale=.7]
      \newcommand{\ca}{-3.5}
      \newcommand{\cb}{4.5}
      \foreach \x in {-3,-1,0,1,4} {
      \draw (\x,\ca)--(\x,\cb); };
      \foreach \y in {-2.8,-0.3, 1.3, 3.4} {
      \draw (\ca,\y)--(\cb,\y); };
      \clip (\ca,\ca) rectangle (\cb,\cb);
      \foreach \z in {-6,-3,-1,0,1,3,6} {
      \draw[xshift=\z cm] (\ca,\cb)--(\cb,\ca); };
    \end{tikzpicture}
    \caption{Doubly and triply infinite configurations in two
    dimensions}
    \label{fig:dt-inf}
  \end{figure}
\end{example}


\section{The hypertoric metric}
\label{sec:class-hypert-manif}

Let us now describe the metric on $M$, using the notation of
\cite{BD}.
Let $\phi : M \rightarrow \iH \otimes (\bR^n)^* \cong \bR^{3n}$ denote the moment map and let $U$ be an open subset of $\iH \otimes (\bR^n)^*$ such that
\( T^n \) acts freely on \( \phi^{-1}(U) \).
The hyperk\"ahler structure on 
\( \phi^{-1}(U) \) is uniquely determined via a \emph{polyharmonic} function \( F \)
on \( U \subset \iH\otimes (\bR^n)^* \). (A function is
polyharmonic if it is harmonic on each affine
subspace \( c + \iH \otimes \bR\alpha \),
\( \alpha \in (\bR^n)^*\setminus\Set0 \)).
 
Explicitly, we write the base $\bR^{3n}$ as $\bR^n \times \bC^n$ and use
$(a,b)$ as coordinates on $\bR^n \times \bC^n$. As we are identifying
$\bR^n$ with its dual, expressions like $a(u_k)$ will be written
$\langle a, u_k \rangle$ etc.

 The hyperKähler metric is of the form
\begin{equation} \label{GH}
g = \sum_{i,j} [ \Phi_{ij} (da_i da_j + db_i d \bar{b}_j) +
(\Phi^{-1})_{ij} (dy_i - A_i)(dy_j - A_j) ]
\end{equation}
for suitable fibre coordinates $y_i$ on the $T^n$ fibres,
 and where $A$ is the connexion form
on the $T^n$-bundle. The pair $(A, \Phi)$ satisfy the generalised monopole
equations of Pedersen and Poon \cite{Pedersen-Poon:Bogomolny}. $\Phi$
is often referred to as the {\em potential} for the metric (if $n=1$
of course our set-up reduces to the Gibbons-Hawking ansatz and
$\Phi$ satisfies the equation $\nabla \Phi = {\rm curl} \; A$).

Moreover the monopole may be described in terms of a {\em prepotential} function
$F$ via the equations
\begin{equation}
(\Phi_{ij}, A_j) = \left( F_{a_i a_j}, \frac{1}{2} 
\sum_{k} \sqrt{-1} (F_{a_j b_k} db_k-
F_{a_j \bar{b}_k} d \bar{b}_k ) \right).
\end{equation}
where $F$ satisfies the polyharmonic equations
$F_{a_i a_j} + F_{b_i \bar{b}_j}=0$. The fiber coordinates $y_i$ may also
be related to $F$ via the formula $dy_i = \bar{\partial}_1 F_{a_i} -
\partial_1 F_{a_i}$ where $\partial_1$ is the Dolbeault operator
corresponding to the choice of complex structure given by the 
 splitting of $\bR^{3n}$ as $\bR^n \times \bC^n$. 

The prepotential
$F$ may be written explicitly as a formal sum as follows. We let
\begin{equation} \label{s}
  s_k = 2 ( \langle a, u_k \rangle - \lambda_k^\bR),
  \end{equation}
\begin{equation} \label{v}
v_k = \langle b,u_k \rangle -  \lambda_k^\bC
\end{equation}
 and define $r_k \geq 0$ by $r_{k}^2 = s_{k}^2 + 4 v_{k}^2$. 
So the distance from
$(a,b)$ to the flat $H_k$ is $r_k / 2 |u_k|$. Now we have
\begin{equation*}
  F = \frac12\sum_{k\in\bL} (s_k\log(s_k+r_k) - r_k).
\end{equation*}
and hence
\begin{equation} \label{Vij}
 \Phi_{ij} = F_{a_i a_j} = \sum_{k \in \bL} \frac{(u_k)_i  (u_k)_j}{r_k}
\end{equation}

Note that our convergence condition (\ref{eq:convergence}) guarantees
convergence of the sum (\ref{Vij}). We also remark that $\Phi$ may be
rewritten as
\[
\sum_{k \in \bL}  \frac{u_k \otimes u_k}{r_k}
\]
Each individual term in the sum is positive semidefinite with kernel
$u_k^\perp$, and as the $u_k$ span $\R^n$, the sum is positive definite
as required.

On the flat $H_k$ we have $s_k = v_k =0$ and hence $r_k=0$, giving a
singularity of $\Phi_{ij}$, but these are just coordinate singularities
corresponding to the fact that the $T^n$ fibres collapse to
lower-dimensional tori over the flats. The metric on $M$ remains
smooth at these points.

If $n=1$, of course, the flats are points in $\R^3$, and our
expression gives the Gibbons-Hawking form of the Anderson-Kronheimer-Lebrun metrics \cite{AndersonMT-KL:infinite}.

\begin{remark}
  The expression $\Phi_{ij} (da_i da_j + db_i d \bar{b}_j)$ gives the
  projected metric on the base $\R^{3n}$ and projection of the
  hypertoric metric to this base is a Riemannian submersion. Similar
  arguments to those in \cite{AndersonMT-KL:infinite} and
  \cite{Goto:A-infinity} show that our hypertoric metrics are
  complete.  For a Cauchy sequence in the hypertoric metric projects
  to a Cauchy sequence in the metric on the base, and comparison with
  the analogous metric with potential $\Phi$ given by a finite sum
  shows that the projected sequence gets trapped in a compact region
  of $\R^{3n}$ and hence converges. As the projection onto $\R^{3n}$
  is proper, this means that the original sequence converges also.
\end{remark}

\begin{remark}
Theorem \ref{classification} stated that a 
 connected hypertoric manifold of
  dimension~\( 4n \) is topologically (up to a product with $(S^1 \times \R^3)$
  factors) a hypertoric manifold of the type constructed
  in this paper, ie. the hyperKähler quotient
  of a flat Hilbert hyperKähler manifold by an Abelian Hilbert Lie
  group.  The hyperKähler metric 
  is either the one induced by the hyperKähler quotient construction 
or a Taub-NUT deformation of this metric.  

In the case of the Taub-NUT deformations the potential is modified
by the addition of extra terms thus:
\begin{equation*}
  F = \sum_{k\in \bL} a_k(s_k\log(s_k+r_k) - r_k) + \sum_{i,j=1}^n
  c_{ij}(4x_ix_j - z_i\overline z_j - z_j\overline z_i)
\end{equation*}
where the \( c_{ij} \) terms give the deformation
of a metric determined by the first sum.
\end{remark}

\section{Periodic examples} \label{sec:periodic}

In order to make the hyperk\"ahler quotient construction work, we
assumed that the $\lambda_k$ grow fast enough with $k$ so that, as in
equation (\ref{eq:convergence}), the sum
$\sum_{k} ( 1 + |\lambda_k|)^{-1}$ is finite.  For example, in
Goto's examples, $\lambda_k$ is taken to grow like $k^2$.  The
$\lambda_k$ of course essentially give the distances of the flats from
the origin.

In the four dimensional case, when the flats are just points in
$\R^3$, Ooguri-Vafa \cite{OV} considered another viewpoint involving
points spaced out at distances growing linearly in $k$--in fact the
points are arranged periodically so the set-up is invariant under $\bZ$
translations. Of course this does not fit into the hyperk\"ahler
quotient picture, but it can be interpreted using the Gibbons-Hawking
form of the metric (\ref{GH}). The points are located at $(k,0,0)$ and
the Ooguri-Vafa potential is now defined to be
\[
  \Phi = \sum_{k} \frac{1}{\sqrt{ (a-k)^2 + b^2 }} - \frac{1}{|k|}.
\]
so the $1/r_{k}$ terms are modified by the subtraction of $1/|k|$ to
ensure convergence. We refer to \cite{Gross-Wilson} for further
background on the Ooguri-Vafa metric. Note that this metric is incomplete
so not included in the classification result Thm. \ref{classification}.

\medskip
We can generalise this idea to the situation in our paper.  
Our potential was
$\Phi_{ij} = \sum_{k=1}^{d} \frac{(u_k)_i
(u_k)_j}{r_k}$
where $r_k = +\sqrt{ s_k^2 + 4 v_k \bar{v}_k }$.

We can now modify this expression in the spirit of Ooguri-Vafa by
changing the $1/r_k$ terms to $1/r_k - 1 /|k|$. We choose the
$\lambda_k$ to have complex part zero.  (Note that the nature of the
singularity of $\Phi_{ij}$ on the flats remains the same, so the 
metric extends smoothly over these points).


Recall that there are only finitely many distinct $u_k$, subject to
the assumptions of Proposition \ref{finite}.
We can now arrange that for each $u_k$ the associated $\lambda$'s just
take all possible values in a translate of $\bZ$. So our diagram of
hyperplanes has periodicity--invariant under a group of translations
isomorphic to $\bZ^n$

This is most clearly seen in terms of the prepotential function $F$
introduced in \ref{sec:class-hypert-manif}, satisfying
$\Phi_{ij} = F_{a_i a_j} = - F_{b_i \bar{b}_j}$.  The basic
four-dimensional example is $F_{0} (x,w) = x \log(x+r)-r$ where
$r^2 = x^2 + 4 w \bar{w}$ giving $\Phi = \frac{1}{r}$.  One can obtain
higher dimensional examples by considering $F(a,b) = F_0 (s_k, v_k)$
where $s_k, v_k$ are given by (\ref{s}),(\ref{v}). Note that the second derivatives for
$F$ now scale by $(u_k)_i (u_k)_j$ so $F$ is still polyharmonic. The linearity of the 
equations for $F$ means that we can superpose solutions to get
\[
  F = \sum_{k} F_0 (s_k, v_k) = \sum_{k} s_k \log(s_k + r_k) -r_k
\]
In the finite case this is just formula (8.2) in
\cite{BD}, and on taking second derivatives yields
the formula of Theorem 9.1 in that paper.  In the infinite case this gives the formula (\ref{Vij}) of
the preceding section of the current paper.

But we can also replace $F_0$ by the corresponding prepotential
$F_{OV}$ for the Ooguri-Vafa metric (which need not be known explicitly). This prepotential is now periodic
in the $x$ variable-- $F_{OV}(x+1, w) = F_{OV}(x,w)$. We can now
consider
\[
  \sum_{k} F_{OV} (s_k, v_k)
\]
where $k$ ranges over a finite set of indices in one-to-one
correspondence with the set of {\em distinct} vectors $u_k$. This
prepotential is invariant under $(a,b) \mapsto (a+ u_j, b)$, as the
$u_k$ are all integral, so the metric we obtain is invariant under the
lattice isomorphic to $\bZ^n$ spanned by the vectors $u_j$. An example
of this construction occurs in the physics literature in \cite{KSZ}.

Note that in the original Ooguri-Vafa metric, the complex coordinate $b$ has
to lie in the unit disc for the correct convergence properties to
hold. In our more general set up, we will need $| v_k | < 1$, for all
$k$, which defines a non-empty neighbourhood $U$ of the origin in
$\bC^n$ as there are only finitely many distinct $u_k$.

As in \S \ref{sec:properties} we can consider projection onto the complex 
coordinate $b$ (this is $a_{\C}$ in the notation of \S \ref{sec:properties}).
If $b$ is in the complement of the complex flats then
the fibre
is a copy of $T^n \times \R^n$.
We can then quotient by $\bZ^n$ and get a picture where the quotient
space fibres over $U \subset \bC^n$ with generic fibres tori $T^{2n}$. 
If $b$ lies in a complex flat, the fibre is obtained from
a $2n$-torus by collapsing some $T^n$ fibres to lower-dimensional tori
over the subtori of the base $T^n$ that are obtained by quotienting 
the real hyperplanes $\langle a_{\bR}, u_k \rangle = \lambda_k^1$ in $\bR^n$.

This
generalises 
the $n=1$ Ooguri-Vafa case where one can quotient by $\bZ$ and get an elliptic
fibration over the complex disc with singular fibre a nodal cubic 
 over the origin. The nodal cubic is of course obtained by pinching
a single $S^1$ in the elliptic curve to a point.

\begin{acknowledgement}
  It is a pleasure to thank Andrew Swann for many useful discussions,
  Melanie Rupflin for advice on monotone operators and Michael
  Thaddeus for suggesting that there should be hypertoric analogues of
  the Ooguri-Vafa metric.
\end{acknowledgement}

\end{document}